\documentclass[11pt]{article}
\usepackage{etex}
\usepackage[all]{xy}
\usepackage{cite}
\usepackage{amsfonts}
\usepackage{amsthm}
\usepackage{enumerate}
\usepackage{graphicx}
\usepackage{mathrsfs}
\usepackage{bm}
\usepackage[utf8]{inputenc}
\usepackage{amssymb,amsmath} 
\usepackage{makecell}
\usepackage{tikz}
\usepackage{pgf}
\usepackage{tikz}
\usetikzlibrary{patterns}
\usepackage{pgffor}
\usepackage{pgfcalendar}
\usepackage{pgfpages}
\usepackage{shuffle,yfonts}
\usepackage{mathtools}
\DeclareFontFamily{U}{shuffle}{}
\DeclareFontShape{U}{shuffle}{m}{n}{ <-8>shuffle7 <8->shuffle10}{}

\newcommand{\nc}{\newcommand}
\nc{\AMZV}{\mathsf {AMZV}}
\nc{\ud}{\mathrm{d}}
\nc{\ES}{\mathsf {ES}}
\nc{\MZV}{\mathsf {MZV}}
\nc{\MtV}{\mathsf {MtV}}
\nc{\MTV}{\mathsf {MTV}}
\nc{\MSV}{\mathsf {MSV}}
\nc{\MMV}{\mathsf {MMV}}
\nc{\MMVo}{\mathsf {MMVo}}
\nc{\MMVe}{\mathsf {MMVe}}
\nc{\AMMV}{\mathsf {AMMV}}
\nc{\AMTV}{\mathsf {AMTV}}
\nc{\AMtV}{\mathsf {AMtV}}
\nc{\AMSV}{\mathsf {AMSV}}
\nc{\CMZV}{\mathsf {CMZV}}
\nc{\sha}{\shuffle}
\nc{\cst}{\rotatebox[origin=c]{180}{$\sha$}}
\nc{\cstt}{\rotatebox[origin=c]{180}{$\scriptstyle \sha$}}
\nc{\de}{\delta}
\nc{\DD}{{\mathbb D}}
\nc{\anbb}[1]{\left\langle#1\right\rangle}
\nc{\bibb}[1]{\left\{#1\right\}}
\nc{\mibb}[1]{\left[#1\right]}
\nc{\smbb}[1]{\left(#1\right)}
\nc{\doubb}[1]{\llbracket#1\rrbracket}
\nc{\dm}[1]{\left|#1\right|}

\nc{\Gbinom}[2]{\genfrac{(}{)}{0mm}{0}{#1}{#2}}
\nc{\gbinom}[2]{\genfrac{(}{)}{0mm}{1}{#1}{#2}}
\nc{\Rbinom}[2]{\genfrac{\langle}{\rangle}{0mm}{0}{#1}{#2}}
\nc{\rbinom}[2]{\genfrac{\langle}{\rangle}{0mm}{1}{#1}{#2}}
\nc{\Qbinom}[2]{\genfrac{[}{]}{0mm}{0}{#1}{#2}_q}
\nc{\qbinom}[2]{\genfrac{[}{]}{0mm}{1}{#1}{#2}_q}

\nc{\binq}[2]{\genfrac{[}{]}{0mm}{0}{#1}{#2}}
\nc{\tbnq}[2]{\genfrac{[}{]}{0mm}{1}{#1}{#2}}
\nc{\cinq}[2]{\genfrac{\{}{\}}{0mm}{0}{#1}{#2}}
\nc{\tcnq}[2]{\genfrac{\{}{\}}{0mm}{1}{#1}{#2}}

\nc{\mfrac}[2]{\genfrac{}{}{0pt}{}{#1}{#2}}
\nc{\tf}{\tfrac}
\nc{\db}{{\mathbb D}}
\nc{\pari}{{\rm par}}
\nc{\dk}{{\mathbb K}}
\nc{\ola}{\overleftarrow}
\nc{\ora}{\overrightarrow}
\nc{\lra}{\longrightarrow}
\nc{\Lra}{\Longrightarrow}
\nc\Res{{\rm Res}}
\nc\setX{{\mathsf{X}}}
\nc\fA{{\mathfrak{A}}}
\nc\evaM{{\texttt{M}}}
\nc\evaML{{\text{\em{\texttt{M}}}}}
\nc\z{{\texttt{z}}}
\nc\emz{\emph{\texttt{z}}}
\nc\tx{{\texttt{x}}}
\nc\txp{{\tx_1}} 
\nc\txn{{\tx_{-1}}} 
\nc\neo{{1}}
\nc{\yi}{{1}}
\nc\one{{-1}}
\nc\gD{{\Delta}}
\nc\eps{{\varepsilon}}
\nc{\bfMB}{{\bf MB}}
\nc{\bftB}{{\bf tB}}
\nc{\bfTB}{{\bf TB}}
\nc{\bfSB}{{\bf SB}}
\nc{\bfB}{{\bf B}}
\nc{\bfp}{{\bf p}}
\nc{\bfq}{{\bf q}}
\nc{\bfr}{{\bf r}}
\nc{\bfu}{{\bf u}}
\nc{\bfv}{{\bf v}}
\nc{\bfw}{{\bf w}}
\nc{\bfy}{{\bf y}}
\nc{\T}{\ddot{t}}
\nc{\bfe}{{\boldsymbol{\sl{e}}}}
\nc{\bfi}{{\boldsymbol{\sl{i}}}}
\nc{\bfj}{{\boldsymbol{\sl{j}}}}
\nc{\bfk}{{\boldsymbol{\sl{k}}}}
\nc{\bfl}{{\boldsymbol{\sl{l}}}}
\nc{\bfm}{{\boldsymbol{\sl{m}}}}
\nc{\bfn}{{\boldsymbol{\sl{n}}}}
\nc{\bfs}{{\boldsymbol{\sl{s}}}}
\nc{\bft}{{\boldsymbol{\sl{t}}}}
\nc{\bfx}{{\boldsymbol{\sl{x}}}}
\nc{\bfz}{{\boldsymbol{\sl{z}}}}
\nc\bfgs{{\boldsymbol \gs}}
\nc\bfgl{{\boldsymbol \lambda}}
\nc\bfsi{{\boldsymbol \gs}}
\nc\bfet{{\boldsymbol \eta}}
\nc\bfeta{{\boldsymbol \eta}}
\nc\bfeps{{\boldsymbol \eps}}
\nc\mmu{{\boldsymbol \mu}}
\nc\bfone{{\bf 1}}
\nc{\myone}{{1}}

 \nc{\calA}{{\mathcal A}}
 \nc{\calB}{{\mathcal B}}
 \nc{\calC}{{\mathcal C}}
 \nc{\calD}{{\mathcal D}}
 \nc{\calE}{{\mathcal E}}
 \nc{\calF}{{\mathcal F}}
 \nc{\calG}{{\mathcal G}}
 \nc{\calH}{{\mathcal H}}
 \nc{\calI}{{\mathcal I}}
 \nc{\calJ}{{\mathcal J}}
 \nc{\calK}{{\mathcal K}}
 \nc{\calL}{{\mathcal L}}
 \nc{\calM}{{\mathcal M}}
 \nc{\calN}{{\mathcal N}}
 \nc{\calO}{{\mathcal O}}
 \nc{\calP}{{\mathcal P}}
 \nc{\calQ}{{\mathcal Q}}
 \nc{\calR}{{\mathcal R}}
 \nc{\calS}{{\mathcal S}}
 \nc{\calT}{{\mathcal T}}
 \nc{\calU}{{\mathcal U}}
 \nc{\calV}{{\mathcal V}}
 \nc{\calW}{{\mathcal W}}
 \nc{\calX}{{\mathcal X}}
 \nc{\calY}{{\mathcal Y}}
 \nc{\calZ}{{\mathcal Z}}
  \nc{\cala}{{\mathcal a}}
 \nc{\calb}{{\mathcal b}}
 \nc{\calc}{{\mathcal c}}
 \nc{\cald}{{\mathcal d}}
 \nc{\cale}{{\mathcal e}}
 \nc{\calf}{{\mathcal f}}
 \nc{\calg}{{\mathcal g}}
 \nc{\calh}{{\mathcal h}}
 \nc{\cali}{{\mathcal i}}
 \nc{\calj}{{\mathcal j}}
 \nc{\calk}{{\mathcal k}}
 \nc{\call}{{\mathcal l}}
 \nc{\calm}{{\mathcal m}}
 \nc{\caln}{{\mathcal n}}
 \nc{\calo}{{\mathcal o}}
 \nc{\calp}{{\mathsf p}}
 \nc{\calq}{{\mathcal q}}
 \nc{\calr}{{\mathcal r}}
 \nc{\cals}{{\mathcal s}}
 \nc{\calt}{{\mathcal t}}
 \nc{\calu}{{\mathcal u}}
 \nc{\calv}{{\mathcal v}}
 \nc{\calw}{{\mathcal w}}
 \nc{\calx}{{\mathcal x}}
 \nc{\caly}{{\mathcal y}}
 \nc{\calz}{{\mathcal z}}
 \nc{\ot}{{\otimes}}
\usetikzlibrary{arrows,shapes,chains}
 \allowdisplaybreaks

\catcode`!=11
\let\!int\int \def\int{\displaystyle\!int}
\let\!lim\lim \def\lim{\displaystyle\!lim}
\let\!sum\sum \def\sum{\displaystyle\!sum}
\let\!sup\sup \def\sup{\displaystyle\!sup}
\let\!inf\inf \def\inf{\displaystyle\!inf}
\let\!cap\cap \def\cap{\displaystyle\!cap}
\let\!max\max \def\max{\displaystyle\!max}
\let\!min\min \def\min{\displaystyle\!min}
\let\!frac\frac \def\frac{\displaystyle\!frac}
\catcode`!=12

\nc{\gam}{{\gamma}}
\nc{\gG}{{\Gamma}}
\nc{\om}{{\omega}}
\nc{\vep}{{\varepsilon}}
\nc{\ga}{{\alpha}}
\nc{\gl}{{\lambda}}
\nc{\gb}{{\beta}}
\nc{\gd}{{\delta}}
\nc{\gf}{{\varphi}}
\nc{\gs}{{\sigma}}
\nc{\gk}{{\kappa}}
\nc{\gS}{\Sigma}
\let\oldsection\section
\renewcommand\section{\setcounter{equation}{0}\oldsection}

\allowdisplaybreaks

\DeclareMathOperator{\Li}{Li}

\nc\UU{\mbox{\bfseries U}}
\nc\FF{\mbox{\bfseries \itshape F}}
\nc\h{\mbox{\bfseries \itshape h}}\nc\dd{\mbox{d}}
\nc\g{\mbox{\bfseries \itshape g}}
\nc\xx{\mbox{\bfseries \itshape x}}

\def\N{\mathbb{N}}
\def\Z{\mathbb{Z}}

\def\ze{\zeta}

\def\xx{\left(\frac{1-x}{1+x} \right)}

\nc\divg{{\text{div}}}
\theoremstyle{plain}
\newtheorem{thm}{Theorem}[section]
\newtheorem{lem}[thm]{Lemma}

\newtheorem{cor}[thm]{Corollary}

\newtheorem{pro}[thm]{Proposition}
\theoremstyle{definition}

\newtheorem{re}[thm]{Remark}

\newtheorem{qu}{Question}[section]

\nc{\cicc}[1]{{}_{{}^{ \bigcirc\hskip-1.2ex{#1}\hskip.3ex{}}}}
\nc{\cic}[1]{{}^{\bigcirc\hskip-1.15ex{\raisebox{-0.015cm}{\text{$\scriptscriptstyle #1$}}}\hskip.25ex{}}}
\nc{\ccic}[1]{{}^{\bigcirc\hskip-1.5ex{\raisebox{-0.015cm}{\text{$\scriptscriptstyle #1$}}}\hskip.25ex{}}}
\nc{\ncic}[1]{ {\bigcirc\hskip-1.6ex{\raisebox{-0.0cm}{\text{$\scriptstyle #1$}}}\hskip.25ex{}}}
\nc{\nncic}[1]{ {\bigcirc\hskip-2ex{\raisebox{-0.0cm}{\text{$\scriptstyle #1$}}}\hskip.25ex{}}}
\nc{\cci}[1]{{}_{{}^{ {\textstyle \bigcirc}\hskip-2.05ex{#1}\hskip-.35ex{}}}}
\nc{\ccicc}[1]{{}_{{}^{ {\textstyle \bigcirc}\hskip-1.55ex{#1}\hskip-0.1ex{}}}}
\nc{\x}{\rm{x}}
\nc{\tworow}[2]{\left(#1 \atop #2\right)}

\nc{\fl}{{\mathfrak l}}
\nc{\fm}{{\mathfrak m}}

\begin{document}
\title{\bf Contour Integration and Cyclotomic Ap\'ery-Like Series Involving Generalized Binomial Coefficients}
\author{
{Ce Xu\thanks{Email: cexu2020@ahnu.edu.cn}}\\[1mm]
\small School of Mathematics and Statistics, Anhui Normal University,\\ \small Wuhu 241002, P.R. China
}
\date{}
\maketitle

\noindent{\bf Abstract.} In this paper, we present a method based on contour integration to investigate a class of cyclotomic parametric Ap\'ery-like series. The general term of such series involves a parametric central binomial coefficient, which is defined via the Gamma function. Using this approach, we express a family of cyclotomic Ap\'ery-like series in terms of multiple polylogarithms, cyclotomic Hurwitz zeta values, Riemann zeta values and $\log(2)$. In particular, we provide several illustrative examples and corollaries, which enable us to recover a number of known results on Ap\'ery-like series. At the same time, we have also left open two questions regarding Ap\'ery-like series. Moreover, by considering integrals of the generating function for Fuss-Catalan numbers, we derive an alternative expression for a classical Ap\'ery-like series. Combining this with known results allows us to establish several identities for multiple polylogarithm functions.

\medskip

\medskip


\section{Introduction}

For $\Re(s) > 1$, the classical \emph{Riemann zeta function} $\zeta(s)$ is defined by the absolutely convergent infinite series:
\begin{align}
\zeta(s):=\sum_{n=1}^\infty \frac1{n^s}.
\end{align}
Since its introduction, the study of the Riemann zeta function has remained at the forefront of mathematical research. In particular, investigating the irrationality and transcendence of its values at odd integers greater than 1 constitutes a profoundly important class of problems in this field. The values of the Riemann zeta function at integers greater than 1 are known as Riemann zeta values. If $s>1$ is an even integer, the value $\zeta(s)$ is referred to as an \emph{even Riemann zeta value}; if $s>1$ is odd, it is called an \emph{odd Riemann zeta value}, respectively. The irrationality and transcendence of even zeta values $\zeta(2k)$ were resolved long ago, as they can be explicitly expressed by Euler's celebrated formula \[\zeta(2k)=\frac{(-1)^{k-1}B_{2k}(2\pi)^{2k}}{2(2k)!}\quad (k\in\N)\]
combined with the transcendence of $\pi$, where $B_k$ denotes the \emph{Bernoulli numbers} defined by
\[\frac{x}{e^x-1}=\sum_{n=0}^{\infty}{\frac{B_n}{n!}x^n}.\]
In contrast, the study of irrationality and transcendence for odd zeta values saw no significant progress for a long time. A major breakthrough came in 1979, when Roger Ap\'ery \cite{Apery1978} proved the irrationality of $\zeta(3)$ by studying the following infinite series involving the central binomial coefficient $\binom{2n}{n}$:
\begin{equation*}
\zeta(3)=\frac52 \sum_{n=1}^\infty \frac{(-1)^{n+1}}{n^3\binom{2n}{n}}.
\end{equation*}
This marked the first important result on the irrationality of odd Riemann zeta values. Due to Ap\'ery's groundbreaking contributions to the irrationality of odd Riemann zeta values, many subsequent researchers have referred to series involving central binomial coefficients as ``Ap\'ery-like series," which has since become a prominent research direction in number theory. It is particularly noteworthy that Professor Zhi-Wei Sun has enthusiastically proposed numerous original conjectures. In his paper \cite{Sun2015} and monograph \cite{Sun2021}, he put forward many conjectural identities concerning Ap\'ery-like series, some of which have already been proven. Currently, the monograph \cite{Sun2021} is available only in Chinese; however, it is understood that Professor Sun is actively working on publishing an English edition. Surprisingly, recent research in this area has revealed profound connections between Ap\'ery-like series and theories of hypergeometric functions, elliptic functions, and multiple zeta values (MZVs for short). Methods derived from these related theories have proven effective in resolving numerous identity problems for Ap\'ery-like series. For some recent work in this direction, we refer the reader to \cite{Au2024,Au2025,Campbell2019,CampbellCA2022,CantariniD2019,CampbellGZ2024,CampbellDS2019,ChenWangZhong2025,ChenWang2025,GR2025,LaiLuorr2022,Lupu2022,WX2021,WLX2022} and the references therein. For a systematic introduction to the theory of multiple zeta values, we recommend Zhao's monograph \cite{Z2016} (the first book dedicated to MZVs theory, which contains nearly all major results prior to 2016) and the recent review article \cite{Z2024}.

Research has shown that among the various methods for investigating the evaluation of infinite series, contour integration and residue computation stand out as highly effective and practical techniques. This approach has led to numerous interesting and celebrated results. For example, Flajolet and Salvy \cite{Flajolet-Salvy} employed the method of contour integration to establish a parity theorem for generalized Euler sums. Recently, the author of this paper, jointly with Rui \cite{Rui-Xu2025}, also applied a similar approach to establish a parity theorem for generalized cyclotomic Euler sums.  Studies presented in \cite{WX2021} and \cite{WLX2022} demonstrate that the contour integration method can also be effectively applied to the investigation of multiple Ap\'ery-like series, leading to a number of interesting conclusions.

The main aim of this paper is to explore the following cyclotomic Ap\'ery-like series, which incorporate generalized binomial coefficients and are termed \emph{cyclotomic parametric Ap\'ery-like series}, using the method of contour integration:
\begin{align}\label{defn-CPAS}
\sum_{n=-\infty}^\infty \frac{4^{n+a}}{(n+a)^q\binom{2n+2a}{n+a}}z^n,
\end{align}
where $a\in \mathbb{C}\setminus \Z,\ q\in\N$ and $z$ is an arbitrary root of unity with $(q,z)\neq(1,1)$. Here the \emph{generalized binomial coefficient} $\binom{a}{b}$ is denoted by
\[\binom{a}{b}:=\frac{\Gamma(a+1)}{\Gamma(b+1)\Gamma(a-b+1)}
    \quad\text{for}\ a,b,a-b\notin \N^-:=\{-1,-2,-3,\ldots\},\]
where $\Gamma(s)$ denotes the \emph{gamma function}, defined for $\Re(s) > 0$ by the improper integral
\begin{align*}
\Gamma(s):=\int_0^\infty e^{-t}t^{s-1}dt.
\end{align*} In particular, when $a$ is a half-integer in \eqref{defn-CPAS} and $2n+2a$  is a non-positive integer, the reciprocal of the parametric central binomial coefficient, $\binom{2n+2a}{n+a}^{-1}$, is interpreted as zero by convention.

\section{Main Theorems and Corollaries}
We will prove that the cyclotomic parametric Ap\'ery-like series \eqref{defn-CPAS} can be expressed as a combination of $\log(2)$, Riemann zeta values, cyclotomic Hurwitz zeta values, and multiple polylogarithms. More precisely, we have the following theorem:
\begin{thm}\label{thm-CPAS} Let $x$ be a root of unity. For $a\in \mathbb{C}\setminus \Z$ and $q\in \N$ with $(q,x)\neq (1,1)$, we have
\begin{align}\label{equ-thm-CPAS}
&\sum_{n=-\infty}^\infty \frac{4^{n+a}}{(n+a)^q\binom{2n+2a}{n+a}}x^{-n}-(-1)^q\left(\Li_1(x;1-a)-x\Li_1(x^{-1};a)\right) \sum_{n=1}^\infty \frac{\binom{2n}{n}}{n^{q-1}4^n}x^n\nonumber\\
&=\sum_{k_1+\cdots+k_5=q-1,\atop \forall k_j\geq 0} \frac{C_{k_1}C_{k_2}D_{k_3}2^{k_3+k_4}}{k_1!k_2!k_3!k_4!}\log^{k_4}(2)\left(x\Li_{k_5+1}(x^{-1};a)-(-1)^{k_5}\Li_{k_5+1}(x;1-a)\right),
\end{align}
where $\Li_{p}(x;b+1)\ ((p,x)\neq (1,1))$ denotes the \emph{cyclotomic Hurwitz zeta function}, defined as follows:
\begin{align}
\Li_{p}(x;b+1):= \sum_{n=1}^\infty \frac{x^n}{(n+b)^p}\quad(p\in\N,\ b\in \mathbb{C}\setminus \N^-)\quad\text{and}\quad \Li_p(x)\equiv \Li_{p}(x;1).
\end{align}
Here
\begin{align}
&C_n:=Y_n\Big(0,1!\zeta(2),-2!\zeta(3),\ldots,(-1)^n(n-1)!\zeta(n)\Big),\label{defn-C-squ}\\
&D_n:=Y_n\Big(0,-1!\zeta(2),2!\zeta(3),\ldots,(-1)^{n-1}(n-1)!\zeta(n)\Big),\label{defn-D-squ}
\end{align}
where $Y_n(x_1,\ldots,x_n)$ stands for the \emph{exponential complete Bell polynomials} (see Section \ref{sectwo}).
\end{thm}
The second cyclotomic Ap\'ery-like series in equation \eqref{equ-thm-CPAS} above can be expressed in terms of multiple polylogarithms and $\log(2)$, as derived from the following proposition (see \cite[Theorem 3.8]{WLX2022}).
\begin{pro}\label{thm-CAS} For $p\in \N$ and $|x|\leq 1$, we have
\begin{align}\label{equ-thm-CAS}
\sum_{n=1}^\infty \frac{\binom{2n}{n}}{n^{p+1}4^n}x^n&=-2\sum_{k=0}^{p-1}\sum_{j+l\leq k,\atop j,l\geq 0} (-1)^{j+l+k}\frac{(j+1)\log^{p-1-k}\Big(\frac{|x|}{4}\Big)}{(p-1-k)!(k-j-l)!}\log^{k-j-l}\left(\frac{|1-\sqrt{1-x}|}{2}\right)\nonumber\\
&\qquad\qquad\times\left( \Li_{l+1,\{1\}_{j+1}}\left(\frac{1-\sqrt{1-x}}{2}\right)-\Li_{l+2,\{1\}_{j}}\left(\frac{1-\sqrt{1-x}}{2}\right)\right),
\end{align}
where $\{1\}_r$ denotes the sequence of $1$ repeated $r$ times. In particular, the reference \cite{Chen2016} provides expressions for the two cases $p=-1$ and $0$:
\begin{align}
&\sum_{n=1}^\infty \frac{\binom{2n}{n}}{4^n}x^n=\frac1{\sqrt{1-x}}-1\quad (|x|\leq 1\quad\text{and}\quad x\neq 1),\label{equ-thm-CAS-zero}\\
&\sum_{n=1}^\infty \frac{\binom{2n}{n}}{n^{}4^n}x^n=2\log\left(\frac{2}{1+\sqrt{1-x}}\right)\quad (|x|\leq 1).\label{equ-thm-CAS-one}
\end{align}
\end{pro}
\begin{re}
In \cite[Theorem 3.8]{WLX2022}, the formula omits two absolute value symbol; specifically $\log^{p-1-k}(x)$ should be written as $\log^{p-1-k}(|x|)$, and $\log^{k-j-l} (1-G_m(x)^{-1})$ should be written as $\log^{k-j-l} (|1-G_m^{-1}(x)|)$. This adjustment is necessary because it can be proven that for $x \in [-(m-1)^{m-1}/m^m, (m-1)^{m-1}/m^m]$, the value of $G_m(x)\ (m\geq 2)$ lies within the interval $(1/2, 2]$. In Section \ref{secfour}, we will employ a method similar to that used in the proof of \cite[Theorem 3.8]{WLX2022} to derive an alternative expression for equation \eqref{equ-thm-CAS}. This approach will allow us to establish several identities involving multiple polylogarithm functions. Certainly, a direct integration approach should yield an alternative expression for equation \eqref{equ-thm-CAS}. As can be readily observed from \eqref{equ-thm-CAS-one}, the series on the left-hand side of \eqref{equ-thm-CAS} can be transformed into the following integral:
\begin{align}
\sum_{n=1}^\infty \frac{\binom{2n}{n}}{n^{p+1}4^n}x^n=2\int_0^x \frac{\log^{p-1}\Big(\frac{x}{t}\Big)\log\left(\frac{2}{1+\sqrt{1-t}}\right)}{t}dt\quad (x\in[-1,1]).
\end{align}
Based on this integral, following a process similar to the proof of \cite[Theorem 2.2]{WX2021}, the desired result should be obtainable.
\end{re}

For any $\bfk=(k_1,\dotsc,k_r)\in\N^r$ and $\bfx=(x_1,\ldots,x_r)\in \mathbb{C}^r$, the classical \emph{multiple polylogarithm function} with $r$-variables $\Li_{\bfk}(\bfx)$ is defined by
\begin{align}\label{defn-mpolyf}
\Li_{\bfk}(\bfx)\equiv \Li_{k_1,\dotsc,k_r}(x_1,\dotsc,x_r):=\sum_{n_1>\cdots>n_r>0} \frac{x_1^{n_1}\dotsm x_r^{n_r}}{n_1^{k_1}\dotsm n_r^{k_r}}
\end{align}
which converges if $|x_1\cdots x_j|<1$ for all $j=1,\dotsc,r$. In particular, if $x_1=x,x_2=\cdots=x_r=1$, then $\Li_{k_1,\ldots,k_r}(x,\{1\}_{r-1})$ is the \emph{single-variable multiple polylogarithm function} $\Li_{k_1,\ldots,k_r}(x)$. Clearly, $\Li_1(x)=-\log(1-x)\ (x\in[-1,1))$.
The \eqref{defn-mpolyf} can be analytically continued to a multi-valued meromorphic function on $\mathbb{C}^r$ (see \cite{KMT2023,Zhao2007d,Zhao2010}). In \cite[Lemma 2.2]{KT2018}, Kaneko and Tsumura provided a detailed explanation of how to analytically continue the single-variable multiple polylogarithm to the domain $\mathbb{C}\setminus[1,+\infty)$ via iterated integration. Starting from the series definition, the multiple polylogarithm function can be expressed as an iterated integral as follows:
\begin{align}
\Li_{k_1,\ldots,k_r}(xx_1,x_2/x_1,\ldots,x_r/x_{r-1})=\int_0^x \left(\frac{dt}{t}\right)^{k_1-1}\frac{x_1dt}{1-x_1t}\cdots \left(\frac{dt}{t}\right)^{k_r-1}\frac{x_rdt}{1-x_rt}.
\end{align}
The theory of iterated integrals was developed firstly by K.T. Chen in the 1960's. It has played important roles in the study of algebraic topology and algebraic geometry in the past half century. Its simplest form is
\begin{align*}
\int_0^x f_1(t)dtf_{2}(t)dt\cdots f_p(t)dt:=&\, \int\limits_{x>t_1>\cdots>t_p>0}f_1(t_1)f_{2}(t_{2})\cdots f_p(t_p)dt_1dt_2\cdots dt_p.
\end{align*}
If $(k_1,\dotsc, k_r)\in\N^r$ and $x_1,\ldots,x_r$ are $N$th roots of unity, then \eqref{defn-mpolyf} become the \emph{cyclotomic multiple zeta values of level $N$} which converges if $(k_1,x_1)\ne (1,1)$ (see \cite{YuanZh2014a} and \cite[Ch. 15]{Z2016}). The level two cyclotomic multiple zeta values are called \emph{alternating multiple zeta values}. When $x=1$ in \eqref{equ-thm-CAS}, the series on the left-hand side was shown in \cite[Theorem 2.2]{WX2021} to be expressible as a combination of alternating multiple zeta values.

Setting $p=1$ and $2$ in \eqref{equ-thm-CAS} give
\begin{align}
&\sum_{n=1}^\infty \frac{\binom{2n}{n}}{n^2 4^n}x^n=2\Li_2\left(\frac{1-\sqrt{1-x}}{2}\right)-\log^2\left(\frac{1+\sqrt{1-x}}{2}\right),\label{case-equ-apery-1}\\
&\sum_{n=1}^\infty \frac{\binom{2n}{n}}{n^3 4^n}x^n=2\log\left(\frac{1+\sqrt{1-x}}{2}\right)\Li_2\left(\frac{1-\sqrt{1-x}}{2}\right)-\frac1{3}\log^3\left(\frac{1+\sqrt{1-x}}{2}\right)\nonumber\\
&\qquad\qquad\qquad+2\Li_{2,1}\left(\frac{1-\sqrt{1-x}}{2}\right)+2\Li_3\left(\frac{1-\sqrt{1-x}}{2}\right),\label{case-equ-apery-2}
\end{align}
where we used the well-known identity (see \cite[Lemma 1(ii)]{AM1999})
\begin{align}\label{equ-p-1-Li}
\Li_{\{1\}_p}(x)=\frac{(-1)^p}{p!}\log^p(1-x)\quad (p\in\N).
\end{align}

Substituting $q=1$ and $2$ into Theorem \ref{thm-CPAS} and applying equations \eqref{equ-thm-CAS-zero} and \eqref{equ-thm-CAS-one}, we obtain the following corollaries.

\begin{cor} Let $x$ be a root of unity and $x\neq 1$. For $a\in \mathbb{C}\setminus \Z$, we have
\begin{align}\label{equ-cor-case-CPAS-one}
\sum_{n=-\infty}^\infty \frac{4^{n+a}}{(n+a)\binom{2n+2a}{n+a}}x^{-n}=\frac{x\Li_1(x^{-1};a)-\Li_1(x;1-a)}{\sqrt{1-x}}.
\end{align}
\end{cor}

\begin{cor} Let $x$ be a root of unity and let $a\in \mathbb{C}\setminus \Z$. Then
\begin{align}\label{equ-cor-case-CPAS}
\sum_{n=-\infty}^\infty \frac{4^{n+a}}{(n+a)^2\binom{2n+2a}{n+a}}x^{-n}&=\Li_2(x;1-a)+x\Li_2(x^{-1};a)\nonumber\\
&\quad-2(\Li_1(x;1-a)-x\Li_1(x^{-1};a))\log(1+\sqrt{1-x}).
\end{align}
\end{cor}
Furthermore, taking the limit as $a\rightarrow 1/2$ in Theorem \ref{thm-CPAS} and noting that for negative integers $n$, the term $\binom{2n+2a}{n+a}^{-1}\rightarrow 0$, while for $n\in \N_0:=\N\cup\{0\}$, $\binom{2n+2a}{n+a}^{-1}\rightarrow (2n+2)\binom{2n+2}{n+1} \pi/16^{n+1}$, we arrive at the following corollary:
\begin{cor}\label{cor-CPAS-casea} Let $x$ be a root of unity. For $q\in\N$ and $(q,x)\neq (1,1)$, we have
\begin{align}\label{equ-cor-CPAS-casea}
&\sum_{n=1}^\infty \frac{n}{(n-1/2)^q}\frac{\binom{2n}{n}}{4^n}x^{1-n}\nonumber\\
&=\sum_{k_1+\cdots+k_5=q-1,\atop \forall k_j\geq 0} \frac{C_{k_1}C_{k_2}D_{k_3}2^{k_3+k_4}}{k_1!k_2!k_3!k_4!}\frac{\log^{k_4}(2)}{\pi}\left(x\Li_{k_5+1}(x^{-1};1/2)-(-1)^{k_5}\Li_{k_5+1}(x;1/2)\right)\nonumber\\
&\quad+\frac{(-1)^q}{\pi}\left(\Li_1(x;1/2)-x\Li_1(x^{-1};1/2)\right) \sum_{n=1}^\infty \frac{\binom{2n}{n}}{n^{q-1}4^n}x^n.
\end{align}
\end{cor}
In particular, setting $x=\pm 1$ in Corollary \ref{cor-CPAS-casea} yields the two known results \cite[Theorems 3.4 and 3.5]{WLX2022}. As two examples, setting $q=1,2$ in \eqref{equ-cor-CPAS-casea} or letting $a=1/2$ in \eqref{equ-cor-case-CPAS-one} and \eqref{equ-cor-case-CPAS} yield
\begin{align}
\sum_{n=1}^\infty \frac{n}{(n-1/2)}\frac{\binom{2n}{n}}{4^n}x^{1-n}=\frac{1}{\pi} \frac{x\Li_1(x^{-1};1/2)-\Li_1(x;1/2)}{\sqrt{1-x}}
\end{align}
and
\begin{align}
\sum_{n=1}^\infty \frac{n}{(n-1/2)^2}\frac{\binom{2n}{n}}{4^n}x^{1-n}&=\frac{\Li_2(x;1/2)+x\Li_2(x^{-1};1/2)}{\pi}\nonumber\\
&\quad-\frac{2}{\pi}(\Li_1(x;1/2)-x\Li_1(x^{-1};1/2))\log(1+\sqrt{1-x}).
\end{align}
Clearly, a closed form for the Cyclotomic Ap\'ery-like series \[\sum_{n=1}^\infty \frac{1}{(n-1/2)^q}\frac{\binom{2n}{n}}{4^n}x^{1-n}\quad (q,x)\neq (1,1)\] can be derived from \eqref{equ-cor-CPAS-casea}.

It should be emphasized that the term $\Li_{m+1}(x;1-a)-(-1)^m x\Li_{m+1}(x^{-1};a)$ appearing in this paper inherently admits an alternative expression, as presented in Theorem \ref{thm-main-two} below.
\begin{thm} (\cite[Thm. 2.8]{Xu2026})\label{thm-main-two} Let $x=e^{i\theta}$ and $\theta\in (0,2\pi)$. For $a\in \mathbb{C}\setminus \Z$ and $m\in \N_0$, we have
\begin{align}\label{equ-thm-main-two}
\Li_{m+1}(x;1-a)-(-1)^m x\Li_{m+1}(x^{-1};a)=\frac{\pi}{m!}\frac{d^m}{da^m}\left(ix^a-\cot(\pi a)x^a\right),
\end{align}
where $i^2=-1$.
\end{thm}

An explicit formula for arbitrary high-order derivatives of $\cot(\pi a)$ can be found in \cite[Thm. 2.3]{Xu18}.
As an example, setting $m=1$ in \eqref{equ-thm-main-two} gives
\begin{align*}
\Li_2(x;1-a)+x\Li_2(x^{-1};a)=\pi^2\csc^2(\pi a)x^a+\pi(i-\cot(\pi a))x^a\ln x.
\end{align*}
Due to the complexity of the higher-order derivative formulas for $\cot(\pi a)$, this paper retains the original form of the left-hand side of \eqref{equ-thm-main-two} rather than rewriting it into its right-hand equivalent.

\section{Preliminaries}\label{sectwo}
In this section, we present some fundamental knowledge and lemmas required for the proof of Theorem \ref{thm-CPAS}.

First, we present the definition of the exponential complete Bell polynomials and provide a brief overview of their key properties. The \emph{exponential complete Bell polynomials} $Y_n(x_1,\ldots,x_n)$ are defined by
\begin{equation}\label{cBell.gf}
\exp\left(\sum_{k=1}^{\infty}x_k\frac{t^k}{k!}\right)
    =\sum_{n=0}^{\infty}Y_n(x_1,x_2,\ldots,x_n)\frac{t^n}{n!}\,,
\end{equation}
and satisfy the recurrence
\begin{equation}\label{cBell.rec}
Y_n(x_1,x_2,\ldots,x_n)=\sum_{j=0}^{n-1}\binom{n-1}{j}x_{n-j}Y_j(x_1,x_2,\ldots,x_j)
    \,,\quad n\geq1\,;
\end{equation}
see \cite[Section 3.3]{C1974} and \cite[Section 2.8]{Riordan58}. From \cite[Eqs. (2.5) and (2.9)]{Xu17.MZVES}, it is known that
\begin{equation}\label{zn*1k}
\int_0^1t^{n-1}\log^k(1-t)dt
    =(-1)^kk!\frac{\zeta_n^\star(\{1\}_k)}{n}\,,
\end{equation}
for $n\geq1$ and $k\geq0$, and $\zeta_n^\star(\{1\}_k)=\frac{1}{k!}Y_k(H_n,1!H_n^{(2)},2!H_n^{(3)},\ldots,(k-1)!H_n^{(k)})$. Here $\zeta^\star_n(\{1\}_k)$ is a special case of the \emph{multiple harmonic star sum}, defined as:
\begin{align}\label{defn-MHSS-1}
\zeta^\star_n(\{1\}_k):=\sum_{n\geq n_1\geq n_2\geq \cdots \geq n_k\geq 1} \frac1{n_1n_2\cdots n_k}\quad (k\in\N_0).
\end{align}
It can be associated with a \emph{multiple harmonic sum}, defined as:
\begin{align}\label{defn-MHS-1}
\zeta_n(\{1\}_k):=\sum_{n\geq n_1>n_2> \cdots > n_k\geq 1} \frac1{n_1n_2\cdots n_k}\quad (k\in\N_0).
\end{align}

We now present several lemmas that will play an essential role in the subsequent proofs.

\begin{lem}(\cite[Lemma 2.1]{Flajolet-Salvy})\label{Lem.Res}
Let $\xi(s)$ be a kernel function and let $r(s)$ be a rational function which is $O(s^{-2})$ at infinity. Then
\begin{equation}\label{Cau.Lind}
\sum_{\alpha \in O}{\rm Res}(r(s)\xi(s),\alpha)
    +\sum_{\beta\in S}{\rm Res}(r(s)\xi(s),\beta)=0\,,
\end{equation}
where $S$ is the set of poles of $r(s)$ and $O$ is the set of poles of $\xi(s)$ that are not poles of $r(s)$. Here ${\rm Res}(h(s),\alpha)$ denotes the residue of $h(s)$ at $s=\alpha$, and the kernel function $\xi(s)$ is meromorphic in the whole complex plane and satisfies $\xi(s)=o(s)$ over an infinite collection of circles $|s|=\rho_k$ with $\rho_k\to+\infty$.
\end{lem}

It is clear that the formula (\ref{Cau.Lind}) is also true if $r(s)\xi(s)=o(s^{-\alpha})$ $(\alpha>1)$ over an infinite collection of circles $|s|=\rho_k$ with $\rho_k\to+\infty$.

\begin{lem}(\cite[Lemma 3.2]{WX2021})\label{Lem.CD}
For $|s|<1$, the following identities hold:
\[
\Gamma(s+1){\rm e}^{\gamma s}=\sum_{n=0}^\infty C_n\frac{s^n}{n!}\quad\text{and}\quad
\{\Gamma(s+1){\rm e}^{\gamma s}\}^{-1}=\sum_{n=0}^\infty D_n\frac{s^n}{n!}\,,
\]
where $\gamma:=\lim_{n\to\infty}\left(\sum_{k=1}^n \frac1{k}-\log n\right)$ is the \emph{Euler-Mascheroni constant}. The definitions of $C_n$ and $D_n$ can be found in \eqref{defn-C-squ} and \eqref{defn-D-squ}, respectively. In particular, we have $(C_k)_{k\in\mathbb{N}_0}=\left(1,0,\ze(2),-2\ze(3),\frac{27}{2}\ze(4),\ldots\right)$ and $(D_k)_{k\in\mathbb{N}_0}=\left(1,0,-\ze(2),2\ze(3),\frac{3}{2}\ze(4),\ldots\right)$.
\end{lem}

\begin{lem}(\cite[Lemma 3.3]{WX2021})\label{Lem.AB}
For nonnegative integer $n$, when $s\to-n$, we have
\begin{align*}
&\Gamma(s){\rm e}^{\gamma(s-1)}
    =\frac{(-1)^n}{n!}{\rm e}^{-\gamma(n+1)}\sum_{k=0}^\infty A_k(n)(s+n)^{k-1}\,,\\
&\frac{1}{\Gamma(s){\rm e}^{\gamma(s-1)}}
    =(-1)^nn!{\rm e}^{\gamma(n+1)}\sum_{k=0}^\infty B_k(n)(s+n)^{k+1}\,,
\end{align*}
where
\[
A_k(n):=\sum_{{k_1+k_2=k,\atop k_1,k_2\geq 0}}
    \zeta_n^\star(\{1\}_{k_1})\frac{C_{k_2}}{k_2!}\,,\quad
B_k(n):=\sum_{{k_1+k_2=k,\atop k_1,k_2\geq 0}}
    (-1)^{k_1}\zeta_n(\{1\}_{k_1})\frac{D_{k_2}}{k_2!}\,.
\]
\end{lem}

The authors of \cite{WX2021,WLX2022} established identities for Ap\'ery-like series by considering the contour integral
\begin{align*}
\oint_{(\infty)} G(s) ds:=\lim_{R\rightarrow \infty}\oint_{C_R} \xi(s) \frac{\Gamma^2(1+s)}{\Gamma(1+2s)}\frac{4^s}{s^q}\quad (q\geq 2)
\end{align*}
and applying the three lemmas above, where $C_R$ denote a circular contour with radius $R$. The \cite{WX2021,WLX2022} the function $\xi(s)$ is primarily chosen from the following types: $1,\ (\psi(-s)+\gamma)^p\ (p\in\N),\ \psi^{(j)}(-s)\ (j\in\N),\ \pi\tan(\pi s),\ \pi\sec(\pi s)$ etc. In this paper, we will choose the aforementioned function $\xi(s)$ to be the \emph{extended trigonometric function} $\Phi(s-a;x)\ (a\in \mathbb{C}\setminus\Z)$, which was recently defined in \cite{Rui-Xu2025} as follows:
\begin{align}\label{defn-Etrigfun}
\Phi(s;x):=\phi(s;x)-\phi\Big(-s;x^{-1}\Big)-\frac1{s}\quad (s\in \mathbb{C}\setminus\Z),
\end{align}
where $\phi(s;x)$ denotes the \emph{generalized digamma function} defined by
\begin{align}
\phi(s;x):=\sum_{k=0}^\infty \frac{x^k}{k+s}\quad (s\notin\N_0^-:=\{0,-1,-2,-3,\ldots\}),
\end{align}
and $x$ is an arbitrary complex number with $|x|\leq 1$ and $x\neq 1$. To ensure convergence, $x$ in \eqref{defn-Etrigfun} must be restricted to roots of unity. Clearly, all integers are poles of $\Phi(s;x)$, and at these poles, it admits the following Laurent expansion:
\begin{lem}(\cite[Eq. (2.6)]{Rui-Xu2025})\label{lem-rui-xu-two} Let $x$ be root of unity. For $n\in \Z$,
\begin{align}\label{LEPhi-function}
\Phi(s;x)=x^{-n} \left(\frac1{s-n}+\sum_{m=0}^\infty \Big((-1)^m\Li_{m+1}(x)-\Li_{m+1}\Big(x^{-1}\Big)\Big)(s-n)^m \right).
\end{align}
\end{lem}
In the next section, we will prove Theorem \ref{thm-CPAS} by considering the contour integral
\begin{align}\label{defn-ContourIntegral}
\oint_{(\infty)} F(s) ds:=\oint_{(\infty)} \Phi(s-a;x) \frac{\Gamma^2(1+s)}{\Gamma(1+2s)}\frac{4^s}{s^q}\quad (a\in\mathbb{C}\setminus\Z,\ q\geq 1).
\end{align}
 For the residue calculations, we also need to provide the following power series expansion of the $\Phi(s-a;x)$.
\begin{lem}\label{lem-extend-xuzhou-two} Let $x$ be root of unity. For $n\in\N_0$ and $a\in \mathbb{C}\setminus \Z$, if $|s+n|<1$, then
\begin{align}\label{equ-extend-xuzhou-two}
\Phi(s-a;x)=x^n\sum_{m=0}^\infty \left((-1)^m \Li_{m+1}(x;1-a)-x \Li_{m+1}\Big(x^{-1};a\Big) \right)(s+n)^m.
\end{align}
\end{lem}
\begin{proof}
The proof of \eqref{equ-extend-xuzhou-two} is straightforward from the definition of $\Phi(s;x)$ and is therefore omitted.
\end{proof}

\section{Proof of Theorem \ref{thm-CPAS} via Contour Integration}

In this section, we proceed to prove Theorem \ref{thm-CPAS} by applying the results established in Section \ref{sectwo}, and subsequently present a more general theorem along with illustrative examples.

First, we show that the contour integral in \eqref{defn-ContourIntegral} is equal to zero.
With the help of the \emph{Legendre duplication formula}
\[
\Gamma(s)\Gamma\left(s+\frac{1}{2}\right)=\sqrt{\pi}\cdot 2^{1-2s}\Gamma(2s)
\]
and the asymptotic expansion for the ratio of two gamma functions
\[
\frac{\Gamma(s+a)}{\Gamma(s+b)}
    =s^{a-b}\left(1+O\Big(\frac{1}{s}\Big)\right)\,,
    \quad\text{for } |\arg(s)|\leq\pi-\varepsilon\,,\ \varepsilon>0\,,\ |s|\to\infty
\]
(see \cite[Sections 2.3 and 2.11]{Luke69.1}), we obtain
\[
\Gamma(s)\Gamma\left(s+\frac{1}{2}\right)=\frac{\sqrt{\pi}}{s^{q-1/2}}\left(1+O\left(\frac{1}{s}\right)\right)\,,\quad |s|\to\infty.
\]
Since the $\Phi(s-a;x)$-function is bounded on the complex plane except at its poles and $\lim_{|s|\rightarrow \infty} |s|^{1-\varepsilon}\Phi(s-a;x)=0\ (\forall \varepsilon>0) $, it follows that for $q\geq 1$
\begin{align*}
\oint_{(\infty)} F(s) ds=\oint_{(\infty)} \Phi(s-a;x) \frac{\Gamma^2(1+s)}{\Gamma(1+2s)}\frac{4^s}{s^q}=0.
\end{align*}
It is easy to see that the integrand has the following poles in the complex plane: $s=0$ is a pole of order $q$, while $s=n+a\ (n\in\Z)$ and $s=-n\ (n\in\N)$ are all simple poles. First, we consider the residue of the integrand at $s=n+a\ (n\in\Z)$. According to Lemma \ref{lem-rui-xu-two}, for $|s+n+a|<1$, we have
\begin{align*}
\Phi(s-a;x)=x^{-n} \left(\frac1{s-a-n}+\sum_{m=0}^\infty \Big((-1)^m\Li_{m+1}(x)-\Li_{m+1}\Big(x^{-1}\Big)\Big)(s-a-n)^m \right).
\end{align*}
Therefore, the residue at this point can be computed as
\begin{align*}
\Res[F(s),n+a]&=\lim_{s\rightarrow n+a} (s-n-a)\Phi(s-a;x) \frac{\Gamma^2(1+s)}{\Gamma(1+2s)}\frac{4^s}{s^q}\\
&=x^{-n} \frac{4^{n+a}}{(n+a)^q\binom{2n+2a}{n+a}}.
\end{align*}
For a negative integer $-n$, by Lemma \ref{Lem.AB}, if $s\to -n$, we have
\[
F(s)=-\Phi(s-a;x)\frac{4^s}{s^q}\frac{(2n-1)!}{(n-1)!^2}
    \sum_{k_1,k_2,k_3\geq0}2^{k_3+1}A_{k_1}(n-1)A_{k_2}(n-1)B_{k_3}(2n-1)
    (s+n)^{k_1+k_2+k_3-1}\,.
\]
Applying Lemma \ref{lem-extend-xuzhou-two}, the residue of the integrand at the simple pole $s=-n\ (n\in\N)$ can be computed as:
\begin{align*}
\Res[F(s),-n]=(-1)^{q+1}\frac{4^{-n}}{n^{q-1}}\binom{2n}{n}x^n\left(\Li_1(x;1-a)-x\Li_1(x^{-1};a)\right).
\end{align*}
Similarly, if $s\to 0$, applying Lemmas \ref{Lem.CD} and \ref{lem-extend-xuzhou-two}, we have
\begin{align*}
F(s)&=\frac{4^s}{s^q}\sum_{m=0}^\infty \left((-1)^m \Li_{m+1}(x;1-a)-x \Li_{m+1}\Big(x^{-1};a\Big) \right)s^m \sum_{k_1,k_2,k_3\geq0}
    \frac{2^{k_3}C_{k_1}C_{k_2}D_{k_3}}{k_1!k_2!k_3!}s^{k_1+k_2+k_3}\\
    &=4^s\sum_{k_1,k_2,k_3,k_4=0}^\infty \frac{2^{k_3}C_{k_1}C_{k_2}D_{k_3}}{k_1!k_2!k_3!}\left((-1)^{k_4} \Li_{k_4+1}(x;1-a)-x \Li_{k_4+1}\Big(x^{-1};a\Big) \right)s^{k_1+k_2+k_3+k_4-q},
\end{align*}
and the residue of the pole of order $q$ at 0 is
\begin{align*}
\Res[F(s),0]&=\frac1{(q-1)!}\lim_{s\rightarrow 0}\frac{d^{q-1}}{ds^{q-1}}\left(s^q\Phi(s-a;x) \frac{\Gamma^2(1+s)}{\Gamma(1+2s)}\frac{4^s}{s^q}\right)\\
&=\sum_{k_1+\cdots+k_5=q-1,\atop \forall k_j\geq 0} \frac{C_{k_1}C_{k_2}D_{k_3}2^{k_3+k_4}}{k_1!k_2!k_3!k_4!}\log^{k_4}(2)\left((-1)^{k_5}\Li_{k_5+1}(x;1-a)-x\Li_{k_5+1}(x^{-1};a)\right).
\end{align*}
By Lemma \ref{Lem.Res}, we have
\begin{align*}
\sum_{n=-\infty}^\infty \Res[F(s),n+a]+\sum_{n=1}^\infty \Res[F(s),-n]+\Res[F(s),0]=0.
\end{align*}
Summing these three contributions gives the statement of the Theorem \ref{thm-CPAS}. \hfill $\square$

It is possible to apply the method presented in this paper to investigate other cyclotomic Ap\'ery-like series. For example, by considering the contour integral
\begin{align}\label{contour-integ-general}
\oint_{(\infty)} \Phi(s-a;x) \frac{\Gamma^2(1+s)}{\Gamma(1+2s)}\frac{4^s}{(s+b-a)^q}ds\quad (a\in\mathbb{C}\setminus\Z,\ b-a\notin\N_0,\ q\geq 2)
\end{align}
one can establish identities involving the Ap\'ery-like series
\begin{align*}
\sum_{n=-\infty}^\infty \frac{4^{n+a}}{(n+b)^q\binom{2n+2a}{n+a}}x^{-n}.
\end{align*}
Specifically, we can present the following general result:
\begin{thm}\label{thm-general-CAS} Let $x$ be a root of unity. For $a,b\in \mathbb{C}\setminus \Z$ and $q\in \N$ with $b-a,2(b-a)\notin \N$ and $(q,x)\neq (1,1)$, we have
\begin{align}\label{equ-general-thm-CPAS}
&\sum_{n=-\infty}^\infty \frac{4^{n+a}}{(n+b)^q\binom{2n+2a}{n+a}}x^{-n}-(-1)^q\left(\Li_1(x;1-a)-x\Li_1(x^{-1};a)\right) \sum_{n=1}^\infty \frac{n\binom{2n}{n}}{(n+a-b)^{q}4^n}x^n\nonumber\\
&=\sum_{k_1+\cdots+k_5=q-1,\atop \forall k_j\geq 0} \frac{C_{k_1}(a-b)C_{k_2}(a-b)D_{k_3}(2a-2b)2^{k_3+k_4}}{k_1!k_2!k_3!k_4!}\log^{k_4}(2)\nonumber\\&\qquad\qquad\qquad\times\left(x\Li_{k_5+1}(x^{-1};b)-(-1)^{k_5}\Li_{k_5+1}(x;1-b)\right).
\end{align}
Here $C_n(x)$ and $D_n(x)$ are expressed in terms of the Gamma function and exponential complete Bell polynomials involving the digamma function and its higher derivatives, defined as follows:
\begin{align}
&C_n(x):=\Gamma(x+1)Y_n\Big(\psi(x+1)+\gamma,\psi^{(1)}(x+1),\psi^{(2)}(x+1),\ldots,\psi^{(n-1)}(x+1)\Big),\label{defn-general-C-squ}\\
&D_n(x):=\frac{Y_n\Big(-\psi(x+1)-\gamma,-\psi^{(1)}(x+1),\psi^{(2)}(x+1),\ldots,-\psi^{(n-1)}(x+1)\Big)}{\Gamma(x+1)},\label{defn-general-D-squ}
\end{align}
where the \emph{digamma function} $\psi(z)$ is defined as
\begin{align}
\psi(z):=\frac{d}{ds}\log\Gamma(z)=-\gamma-\frac1{z}+\sum_{n=1}^\infty \left(\frac1{n}-\frac1{n+z}\right)\quad (z\in \mathbb{C}\setminus \N_0^-).
\end{align}
\end{thm}
\begin{proof}
The proof of this theorem is entirely analogous to that of Theorem \ref{thm-CPAS} above. The key lies in considering the residue calculation of the contour integral \eqref{contour-integ-general}, and this residue computation requires the generalization of Lemma \ref{Lem.CD} into the following form:
\[
\Gamma(s+1){\rm e}^{\gamma s}=\sum_{n=0}^\infty C_n(a)\frac{(s-a)^n}{n!}\quad\text{and}\quad
\{\Gamma(s+1){\rm e}^{\gamma s}\}^{-1}=\sum_{n=0}^\infty D_n(a)\frac{(s-a)^n}{n!}.
\]
The above two power series expansions can be directly obtained using \emph{Fa\`a di Bruno's formula} (see \cite[Sec. 3.4]{C1974}).

Let $H(s)$ be defined as the integrand in the integral in \eqref{contour-integ-general}. It is not difficult to observe that the function \( H(s) \) has the following poles in the complex plane: $a - b $ is a pole of order $q$, $n + a$ (where $n \in \mathbb{Z}$) are simple poles, and $-n$ (where $n \in \mathbb{N}$) are also simple poles.

From Lemma \ref{lem-extend-xuzhou-two} and the two expansions above, we can provide the Laurent series expansion of $H(s)$ at $s = a - b$ as follows:
\begin{align*}
H(s)&=4^s\sum_{k_1,k_2,k_3,k_4=0}^\infty \frac{2^{k_3}C_{k_1}(a-b)C_{k_2}(a-b)D_{k_3}(2a-2b)}{k_1!k_2!k_3!}\nonumber\\
 &\qquad\qquad\times \left((-1)^{k_4}\Li_{k_4+1}(x;1-b)-x \Li_{k_4+1}\Big(x^{-1};b\Big) \right)(s+b-a)^{k_1+k_2+k_3+k_4-q}.
\end{align*}
Therefore, the residue at $s = a - b$ can be obtained through direct computation as follows:
\begin{align*}
\Res[H(s),a-b]&=\frac1{(q-1)!}\lim_{s\rightarrow 0}\frac{d^{q-1}}{ds^{q-1}}\left((s+b-a)^q\Phi(s-a;x) \frac{\Gamma^2(1+s)}{\Gamma(1+2s)}\frac{4^s}{(s+b-a)^q}\right)\\
&=\sum_{k_1+\cdots+k_5=q-1,\atop \forall k_j\geq 0} \frac{C_{k_1}(a-b)C_{k_2}(a-b)D_{k_3}(2a-2b)2^{k_3+k_4}}{k_1!k_2!k_3!k_4!}\log^{k_4}(2)\\
&\qquad\qquad\times\left((-1)^{k_5}\Li_{k_5+1}(x;1-b)-x\Li_{k_5+1}(x^{-1};b)\right).
\end{align*}
Additionally, the residues at the simple poles $n + a$ ($n \in \mathbb{Z}$) and $-n$ ($n \in \mathbb{N}$) can be directly calculated as follows:
\begin{align*}
\Res[H(s),n+a]&=\lim_{s\rightarrow n+a} (s-n-a)\Phi(s-a;x) \frac{\Gamma^2(1+s)}{\Gamma(1+2s)}\frac{4^s}{(s+b-a)^q}\\
&=x^{-n} \frac{4^{n+a}}{(n+b)^q\binom{2n+2a}{n+a}}
\end{align*}
and
\begin{align*}
\Res[H(s),-n]=(-1)^{q+1}\frac{n4^{-n}}{(n+a-b)^{q}}\binom{2n}{n}x^n\left(\Li_1(x;1-a)-x\Li_1(x^{-1};a)\right).
\end{align*}
Applying Lemma \ref{Lem.Res} yields
\begin{align*}
\sum_{n=-\infty}^\infty \Res[H(s),n+a]+\sum_{n=1}^\infty \Res[H(s),-n]+\Res[H(s),a-b]=0.
\end{align*}
Summing these three contributions gives the statement of the Theorem \ref{thm-CPAS}.
\end{proof}

Obviously, letting $b=a$ in Theorem \ref{thm-general-CAS} yields Theorem \ref{thm-CPAS}. From Theorem \ref{thm-general-CAS}, many other conclusions can also be derived. For example, taking $q=1$ and $b\rightarrow a+1/2$, and noting the fact that $\lim_{x\rightarrow -1}D_0(x)=0$, we obtain
\begin{align}
\sum_{n=-\infty}^\infty \frac{4^{n+a}x^{-n}}{(n+a+1/2)\binom{2n+2a}{n+a}}=\frac{x \Li_1(x^{-1};a)-\Li_1(x;1-a)}{\pi} \frac{\Li_1(x;1/2)-x\Li_1(x^{-1};1/2)}{\sqrt{1-x^{-1}}}.
\end{align}
Finally, we conclude this section by proposing the following two questions:
\begin{qu}
In the cyclotomic Ap\'ery-like series given by \eqref{equ-cor-CPAS-casea}, $x$ must be a root of unity. For the more general cyclotomic Ap\'ery-like series
\[\sum_{n=1}^\infty \frac{n}{(n-1/2)^q}\frac{\binom{2n}{n}}{4^n}x^n\quad (|x|\leq 1),\]
is it possible to derive similar formulas? In other words, can such series be expressed as rational linear combinations of Riemann zeta values, $\log(2)$, and multiple polylogarithms?
\end{qu}

\begin{qu} Does the second parametric Ap\'ery-like series
\[\sum_{n=0}^\infty \frac{\binom{2n}{n}x^n}{(n+a)^q 4^n}\quad (a\in \mathbb{C}\setminus \N_0^-)\]
appearing in Theorem \ref{thm-general-CAS} admit a closed-form expression? For instance, can it be expressed in terms of functions such as the digamma function, Gamma function, multiple Hurwitz zeta functions, etc.? Indeed, by employing the method of contour integration or the integral representation of the Beta function, the result for $x=1$ can be derived as follows:
\begin{align}
\sum_{n=0}^\infty \frac{\binom{2n}{n}}{(n+a) 4^n}=\frac{\Gamma^2(a)4^a}{2\Gamma(2a)}.
\end{align}
By successively differentiating the above expression with respect to $a$, a general formula for $x=1$ and $q\in \N$ can be derived. However, does a similar result hold for the case $x\neq 1$?
\end{qu}

\section{Alternative Expression for a Class of Euler-Ap\'ery-like Series}\label{secfour}

In this section, we employ a method analogous to that in \cite[Theorem 3.8]{WLX2022} to derive an alternative expression of the theorem itself. We begin by presenting an introduction to Fuss-Catalan numbers and the relevant properties of their generating function.

The \emph{Fuss-Catalan numbers} (\cite{Av2008, SR2005, HP2001}) are defined as
\begin{align*}
F_m(n):=\frac{1}{(m-1)n+1}\binom{mn}{n} \quad (m\geq 1,n\geq 0)
\end{align*}
and satisfy the generating function relation \cite[(2.9)]{HP2001}
\begin{align}\label{relation-gener-fuss}
G_m(x)=1+xG_m(x)^m\quad\text{and}\quad G_m(0):=1,
\end{align}
where
\begin{align*}
G_m(x):=\sum_{n=0}^\infty F_m(n)x^n\quad \left(|x|\leq \frac{(m-1)^{m-1}}{m^m}\right)
\end{align*}
and it is defined that $0^0:=1$, with the condition that $x\in(-1,1)$ if $m=1$. Clearly, $G_1(x)=(1-x)^{-1}\ (x\in(-1,1))$.

In \cite[Section 2.2]{WLX2022}, the authors established that the function $G_m(x)$ is monotonically increasing on the interval $[-(m-1)^{m-1}/{m^m},(m-1)^{m-1}/{m^m}]$ and satisfies $0<G_m(x)\leq m/(m-1)\ (m\geq 2)$. In fact, from the generating function relation \eqref{relation-gener-fuss}, it is straightforward to observe that $G_m>1/2$.  This follows from a contradiction argument: if we assume $0<G\leq 1/2$ for $x\in [-(m-1)^{m-1}/{m^m},0)$, the identity \eqref{relation-gener-fuss} $G_m(x)=1+xG_m(x)^m$ would imply $G_m(x)>1/2$, which contradicts the assumption. Therefore, for any integer $m>1$, the function $G_m(x)$ takes values in the interval $(1/2,2]$.

The authors of \cite{WLX2022} utilized the generating function relation \eqref{relation-gener-fuss} to prove the following Euler-Ap\'ery-like series identity:
\begin{thm}(\cite[Theorem 3.8]{WLX2022})\label{thm-Bfc} For positive integers $m$ and $p$,
\begin{align}\label{FSN3}
\sum_{n=1}^\infty \binom{mn}{n}\frac{x^n}{n^{p+1}}&=-m\sum_{k=0}^{p-1} \sum_{0\leq j+l\leq k\atop j,l\geq 0}  (-1)^{j+l+k} \frac{(j+1)(m-1)^j}{(p-1-k)!(k-j-l)!}\nonumber\\&\quad\quad\times \log^{p-1-k}(|x|)\log^{k-j-l}(|1-G_m(x)^{-1}|)\nonumber\\&\quad\quad\times \left\{(m-1)\Li_{l+1,\{1\}_{j+1}}(1-G_m(x)^{-1})-\Li_{l+2,\{1\}_{j}} (1-G_m(x)^{-1}) \right\},
\end{align}
where $G_m(x)^{-1}=1/G_m(x)$ and $|x|\leq(m-1)^{m-1}/m^m$.
\end{thm}

Next, we employ a method analogous to the proof of \cite[Theorem 3.8]{WLX2022} to derive an alternative expression for \eqref{FSN3}.
\begin{thm}\label{thm-AFSN3} For positive integers $m>1$ and $p$,
\begin{align}\label{AFSN3}
\sum_{n=1}^\infty \binom{mn}{n}\frac{(-x)^n}{n^{p+1}}&=-\sum_{k=0}^{p-1} \sum_{0\leq j+l\leq k\atop j,l\geq 0}  (-1)^{l+k} \frac{(j+1)m^{j+1}}{(p-1-k)!(k-j-l)!}\nonumber\\&\quad\quad\times \log^{p-1-k}(|x|)\log^{k-j-l}(|1-G_m(-x)|)\nonumber\\&\quad\quad\times \left\{\Li_{l+2,\{1\}_{j}}(1-G_m(-x))+m \Li_{l+1,\{1\}_{j+1}} (1-G_m(-x)) \right\},
\end{align}
where $|x|\leq(m-1)^{m-1}/m^m$.
\end{thm}
\begin{proof}
From \cite[Eq. (3.11)]{WLX2022}, we have
\begin{align*}
\sum_{n=1}^\infty \binom{mn}{n}\frac{(-x)^n}{n}=m\log(G_m(-x)).
\end{align*}
Applying the elementary integral formula
\[\int_0^x t^{n-1}\log^{p-1}\left(\left\vert\frac{x}{t}\right\vert\right)dt=(p-1)!\frac{x^n}{n^{p}}\quad (p,n\in\N)\]
to the expression above yields
\begin{align}\label{proof-Apery-one}
\sum_{n=1}^\infty \binom{mn}{n}\frac{(-x)^n}{n^{p+1}}&=\frac{m}{(p-1)!}\int_0^x \frac{\log^{p-1}\left(\left\vert\frac{x}{t}\right\vert\right)}{t}\log(G_m(-t))dt\nonumber\\
&=\frac{m}{(p-1)!} \sum_{k=0}^{p-1} (-1)^k \binom{p-1}{k} \log^{p-1-k}(|x|) \int_0^x \frac{\log^k(|t|)\log(G_m(-t))}{t}dt.
\end{align}
Performing the variable substitution $u=G_m(-x)\in (1/2,2]$ and $w=G_m(-t)$ in the above integral and using \eqref{relation-gener-fuss}, we obtain
\begin{align}\label{proof-Apery-two}
&\int_0^x \frac{\log^k(|t|)\log(G_m(-t))}{t}dt\nonumber\\
&=\int_1^u \log^k\left(\left\vert\frac{1-w}{w^m}\right\vert  \right)\log(w)\left(\frac1{w-1}-\frac{m}{w}\right)dw\quad (t=1-w)\nonumber\\
&=\int_0^{1-u} \left(\frac{\log^k\left(\left\vert \frac{t}{(1-t)^m} \right\vert\right)\log(1-t)}{t}+m \frac{\log^k\left(\left\vert \frac{t}{(1-t)^m} \right\vert\right)\log(1-t)}{1-t}\right)dt\nonumber\\
&=\sum_{j=0}^k (-m)^j\binom{k}{j}\int_0^{1-u} \left(\frac{\log^{k-j}(|t|)\log^{j+1}(1-t)}{t}+\frac{\log^{k-j}(|t|)\log^{j+1}(1-t)}{1-t} \right)dt\nonumber\\
&=\sum_{j=0}^k (-m)^j\binom{k}{j} \left((-1)^{j+1}(j+1)!\sum_{n=1}^\infty \frac{\zeta_{n-1}(\{1\}_{j})}{n} \int_0^{1-u} t^{n-1} \log^{k-j}(|t|)dt\atop +m(-1)^{j+1}(j+1)!\sum_{n=1}^\infty \zeta_{n-1}(\{1\}_{j+1})\int_0^{1-u} t^{n-1} \log^{k-j}(|t|)dt\right)\nonumber\\
&=-k!\sum_{0\leq j+l\leq k,\atop j,l\geq 0} \frac{(-1)^l(j+1)m^j}{(k-j-l)!}\log^{k-j-l}(|1-u|)\nonumber\\
&\qquad\qquad\qquad\times\left(\sum_{n=1}^\infty \frac{\zeta_{n-1}(\{1\}_j)}{n^{l+2}}(1-u)^n+m \sum_{n=1}^\infty \frac{\zeta_{n-1}(\{1\}_{j+1})}{n^{l+1}}(1-u)^n \right)\nonumber\\
&=-k!\sum_{0\leq j+l\leq k,\atop j,l\geq 0} \frac{(-1)^l(j+1)m^j}{(k-j-l)!}\log^{k-j-l}(|1-u|)\left(\Li_{l+2,\{1\}_j}(1-u)+m\Li_{l+1,\{1\}_{j+1}}(1-u)\right),
\end{align}
where we used the definition of multiple polylogarithm function and the following identities (see \cite{Xu17.MZVES})
\begin{align*}
&\log ^k(1 - x) = (- 1)^kk!\sum\limits_{n = 1}^\infty  \frac{{{x^n}}}{n}{\zeta _{n - 1}}(\{1\}_{k - 1})\quad (x\in[-1,1)),\\
&\int\limits_0^x t^{n-1}\log^m(|t|) dt = \sum\limits_{l = 0}^m {l!\binom{m}{l}\frac{{\left( { - 1} \right)}^l}{n^{l + 1}}x^{n}\log^{m - l} (|x|)} .
\end{align*}
Therefore, substituting \eqref{proof-Apery-two} into \eqref{proof-Apery-one} and performing elementary calculations yields the desired formula.
\end{proof}

Setting $p=1$ and $m=2$ in Theorem \ref{thm-AFSN3}, respectively, we obtain the following two corollaries.
\begin{cor}\label{cor-cthm-ases-p-1} For positive integer $m>1$, we have
\begin{align}\label{equ-cor-sec4-one}
\sum_{n=1}^\infty \frac{\binom{mn}{n}}{n^2}(-1)^nx^n=-m\Li_2(1-G_m(-x))-\frac{m^2}{2}\log^2(G_m(-x)).
\end{align}
\end{cor}

\begin{cor} For positive integer $p$ and $|x|\leq 1$, we have
\begin{align}\label{equ-cor-anotherexa}
\sum_{n=1}^\infty \frac{\binom{2n}{n}}{n^{p+1}}\frac{(-x)^n}{4^n}&=-\sum_{k=0}^{p-1} \sum_{0\leq j+l\leq k\atop j,l\geq 0} \frac{(-1)^{l+k}(j+1)2^{j+1}}{(p-1-k)!(k-j-l)!}\log^{p-1-k}\left(\frac{|x|}{4}\right)\log^{k-j-l}\left(\left\vert\frac{\sqrt{1+x}-1}{\sqrt{1+x}+1}\right\vert\right)\nonumber\\
&\quad\times\left(\Li_{l+2,\{1\}_j}\left(\frac{\sqrt{1+x}-1}{\sqrt{1+x}+1}\right)+2\Li_{l+1,\{1\}_{j+1}}\left(\frac{\sqrt{1+x}-1}{\sqrt{1+x}+1}\right)\right).
\end{align}
\end{cor}

\begin{re}
Corollary \ref{cor-cthm-ases-p-1} also holds for $m=1$ and $x \in (-1/2, 1)$, because setting $m=1$ yields the following known result (see \cite[Page 106, Eq. (2.6.15)]{A2000}):
\begin{align}\label{cor-one-case-1}
2\Li_2(-x)+2\Li_2\left(\frac{x}{1+x}\right)+\log^2(1+x)=0.
\end{align}
In fact, Theorem \ref{thm-AFSN3} also holds for $m=1$ and $x \in (-1/2, 1)$, which can be directly seen from the proof. This is because when $m=1$, only $x\in (-1/2, 1)$ ensures that the range of $u = G_1(-x) = (1+x)^{-1}\in (0, 2)$ (more precisely, $(1/2, 2)$), thereby allowing the use of the power series expansion of $\log^j(1-t)$ in the calculation of \eqref{proof-Apery-two}.
\end{re}

Letting $p=1$ and $2$ in \eqref{equ-cor-anotherexa} and applying \eqref{equ-p-1-Li} give
\begin{align}
&\sum_{n=1}^\infty \frac{\binom{2n}{n}}{n^{2}}(-1)^n\frac{x^n}{4^n}=-2\Li_{2}\left(\frac{\sqrt{1+x}-1}{\sqrt{1+x}+1}\right)-2\log^2\left(\frac{2}{1+\sqrt{1+x}}\right),\label{case2-equ-apery-1}\\
&\sum_{n=1}^\infty \frac{\binom{2n}{n}}{n^{3}}(-1)^n\frac{x^n}{4^n}=2\log\left(\frac{4}{|x|}\left\vert\frac{\sqrt{1+x}-1}{\sqrt{1+x}+1}\right\vert\right)\left(\Li_{2}\left(\frac{\sqrt{1+x}-1}{\sqrt{1+x}+1}\right)
+\log^2\left(\frac{2}{1+\sqrt{1+x}}\right)\right)\nonumber\\
&\qquad\qquad\qquad-2\Li_{3}\left(\frac{\sqrt{1+x}-1}{\sqrt{1+x}+1}\right)+4\Li_{2,1}\left(\frac{\sqrt{1+x}-1}{\sqrt{1+x}+1}\right)
-\frac{8}{3}\log^3\left(\frac{2}{1+\sqrt{1+x}}\right).\label{case2-equ-apery-2}
\end{align}

Obviously, the formulas in Theorems \ref{thm-Bfc} and \ref{thm-AFSN3} must be essentially equal. Therefore, some identities involving multiple polylogarithm functions can be established from the two formulas in the theorems. For instance, from \eqref{case-equ-apery-1} and \eqref{case2-equ-apery-1}, and from \eqref{case-equ-apery-2} and \eqref{case2-equ-apery-2} respectively, we obtain
\begin{align}\label{case-speci-x}
2\Li_2\left(\frac{1-\sqrt{1-x}}{2}\right)+2\Li_{2}\left(\frac{\sqrt{1-x}-1}{\sqrt{1-x}+1}\right)+\log^2\left(\frac{1+\sqrt{1-x}}{2}\right)=0
\end{align}
and
\begin{align}
&2\Li_{2,1}\left(\frac{1-\sqrt{1-x}}{2}\right)-4\Li_{2,1}\left(\frac{\sqrt{1-x}-1}{\sqrt{1-x}+1}\right)\nonumber\\
&=2\log\left(\frac{4}{|x|}\left\vert\frac{\sqrt{1-x}-1}{\sqrt{1-x}+1}\right\vert\right)\left(\Li_{2}\left(\frac{\sqrt{1-x}-1}{\sqrt{1-x}+1}\right)
+\log^2\left(\frac{1+\sqrt{1-x}}{2}\right)\right)\nonumber\\
&\quad-2\Li_3\left(\frac{1-\sqrt{1-x}}{2}\right)-2\Li_{3}\left(\frac{\sqrt{1-x}-1}{\sqrt{1-x}+1}\right)+3\log^3\left(\frac{1+\sqrt{1-x}}{2}\right).
\end{align}
It is evident that \eqref{cor-one-case-1} and \eqref{case-speci-x} represent fundamentally the same equation.
By assigning specific values to $x$ in the two preceding equations, several classical results can be derived. For example, setting $x=1$ in \eqref{case-speci-x} yields
\[\Li_2(1/2)=\frac{\pi^2}{12}-\frac1{2}\log^2(2).\]

{\bf Declaration of competing interest.}
The author declares that he has no known competing financial interests or personal relationships that could have
appeared to influence the work reported in this paper.

{\bf Data availability.}
No data was used for the research described in the article.

{\bf Acknowledgments.}  The author is supported by the General Program of Natural Science Foundation of Anhui Province (Grant No. 2508085MA014).


\medskip

\begin{thebibliography}{99}

\bibitem{A2000}
G.E.\ Andrews, R.\ Askey and R.\ Roy, \emph{Special Functions},
Cambridge University Press, 2000, pp.\ 104-106.

\bibitem{Apery1978}
R.\ Ap\'ery, Irrationalit\'e de $\zeta(2)$ et $\zeta(3)$ (in French),
\emph{Ast\'erisque} \textbf{61}(1979), pp.\ 11--13.

\bibitem{AM1999}
T. Arakawa and M. Kaneko, Multiple zeta values, poly-Bernoulli numbers, and related zeta functions, \emph{Nagoya Math. J.} \textbf{153}(1999), pp.\ 189-209.

\bibitem{Au2024}
K. Au, Multiple zeta values, WZ-pairs and infinite sums computations, \emph{Ramanujan J.} \textbf{66}(3)(2024).

\bibitem{Au2025}
K. Au, Wilf-Zeilberger seeds and non-trivial hypergeometric identities, \emph{J. Symbolic Comput.} 130: 102421.

\bibitem{Av2008}
J.C. Aval, Multivariate Fuss-Catalan numbers, \emph{Discrete Math.} \textbf{308}(2008), pp.\ 4660--4669.

\bibitem{Campbell2019}
J.M.\ Campbell, Series containing squared central binomial coefficients and alternating harmonic
numbers,\emph{ Mediterr. J. Math.} \textbf{16}(2019) Art. 37, 7.

\bibitem{CampbellCA2022}
J.M.\ Campbell, M.\ Cantarini and J.\ D'Aurizio,
Symbolic computations via Fourier--Legendre expansions and fractional operators,
\emph{Integral Transforms Special Func.} \textbf{33}(2)(2022), pp.\ 1--19.

\bibitem{CampbellDS2019}
J.M.\ Campbell, J.\ D'Aurizio and J.\ Sondow,
On the interplay among hypergeometric functions, complete elliptic integrals, and Fourier--Legendre expansions,
\emph{J. Math. Anal. Appl.} \textbf{479}(2019), pp.\ 90--121.

\bibitem{CampbellGZ2024}
J.M.\ Campbell, M.L.\ Glasser and Y. Zhou,
New evaluations of inverse binomial series via cyclotomic multiple zeta values, \emph{SIGMA} \textbf{20}(2024), 079, 14 pages.

\bibitem{CantariniD2019}
M.\ Cantarini and J.\ D'Aurizio,
On the interplay between hypergeometric series, Fourier--Legendre expansions and Euler sums
\emph{Boll.\ Unione Mat.\ Ital.}, \textbf{12}(2019), pp.\ 623--656.

\bibitem{Chen2016}
H.\ Chen, Interesting series associated with central binomial coefficients, Catalan numbers and harmonic
numbers, \emph{J. Integer Seq.} \textbf{19}(2016). Article 16.1.5.


\bibitem{ChenWangZhong2025}
A. Chen, W. Wang and J. Zhong, Evaluations of some double Ap\'ery-type series via Fourier-Legendre expansions, \emph{Bull. Malays. Math. Sci. Soc.} \textbf{48}(4)(2025),  Paper No. 115, 25 pp.

\bibitem{ChenWang2025}
X. Chen and W. Wang, Ap\'ery-type series via colored multiple zeta values and Fourier-Legendre series expansions, \emph{J. Symbolic Comput.} 134:102508.

\bibitem{C1974}
L. Comtet, \emph{Advanced Combinatorics}, D. Reidel Publishing Co., Dordrecht, 1974.

\bibitem{Flajolet-Salvy}
P. Flajolet and B. Salvy, Euler sums and contour integral representations, \emph{Experiment. Math.} \textbf{7}(1)(1998), pp.\  15-35.

\bibitem{SR2005}
S. Fomin and N. Reading, Generalized cluster complexes and coxeter combinatorics, \emph{Int. Math. Res. Notices}, {\bf 44}(2005), pp.\ 2709--2757.

\bibitem{HP2001}
P. Hilton and J. Pedersen, Catalan numbers, their generalizations, and their uses, \emph{Math. Intelligencer}, {\bf 13}(2001), pp.\ 64--75.

\bibitem{GR2025}
M. Gen\v{c}ev and P. Rucki, On a class of multiple Ap\'ery-like series and their reduction, \emph{Mediterr. J. Math.}, (2025)22:195.

\bibitem{KT2018}
M. Kaneko and H. Tsumura, Multi-poly-Bernoulli numbers and related zeta functions, \emph{Nagoya Math. J.} {\bf 232}(2018), pp.\ 19-54.

\bibitem{KMT2023} Y. Komori, K. Matsumoto and H. Tsumura, \textit{The theory of zeta-functions of root systems}, Springer Monographs in Mathematics, Springer Singapore, 2023, pp. 3--5.

\bibitem{LaiLuorr2022}
L. Lai, C. Lupu and D. Orr, Elementary proofs of Zagier's formula for multiple zeta values and its odd variant, \emph{Proc. Amer. Math. Soc.} \textbf{154}(1)(2026), pp.\ 11-24.

\bibitem{Luke69.1}Y.L. Luke,
The Special Functions and Their Approximations, Vol. I.
Mathematics in Science and Engineering, Vol. 53,
Academic Press, New York-London, 1969.

\bibitem{Lupu2022}
C. Lupu, Another look at Zagier's formula for multiple zeta values involving Hoffman elements, \emph{Math. Zeit.} \textbf{301}(2022), pp.\ 3127-3140.

\bibitem{Riordan58}
J. Riordan, \emph{An Introduction to Combinatorial Analysis}, Reprint of the 1958 original, Dover Publications, Inc., Mineola, NY, 2002.

\bibitem{Rui-Xu2025}
H. Rui and C. Xu, Contour integrations and parity results of cyclotomic Euler sums and multiple polylogarithm function, \emph{J.\ Number Theory}, \textbf{283}(2026), pp.\ 64--96.

\bibitem{Sun2015}
Z.-W.\ Sun, New series for some special values of $L$-functions, \emph{Nanjing Univ. J. Math. Biquarterly} \textbf{32}(2015), no.2, pp.\ 189-218.

\bibitem{Sun2021}
Z.-W.\ Sun, \emph{New Conjectures in Number Theory and Combinatorics}, Harbin Institute of Technology Press, (in Chinese) 2021.

\bibitem{WX2021}
W. Wang and C. Xu, Alternating multiple zeta values, and explicit formulas of some Euler-Ap\'ery type series, \emph{Eur. J. Combin.} {\bf 93}(2021), 103283.

\bibitem{WLX2022}
Y. Wang, Y. Li and C. Xu, Evaluations of some Euler-Ap\'ery-type series, \emph{Indian J. Pure Appl. Math.} \textbf{53}(4)(2022), pp.\ 849--864.

\bibitem{Xu17.MZVES}
C. Xu, Multiple zeta values and Euler sums, \emph{J. Number Theory} \textbf{177}(2017), pp.\ 443--478.

\bibitem{Xu18}
C. Xu, Identities about level 2 Eisenstein series, \emph{Commun. Korean Math. Soc.}  \textbf{35}(1)(2020), pp. \ 63--81.

\bibitem{Xu2026}
C. Xu, Symmetry results for cyclotomic multiple Hurwitz zeta values via contour integrals, arXiv:2602.10391.

\bibitem{YuanZh2014a}
H.\ Yuan and J.\ Zhao, Double shuffle relations of double zeta values
and double Eisenstein series of level $N$,
\emph{J.\ London Math.\ Soc.}  \textbf{92}(2)(2015), pp. \ 520--546.

\bibitem{Zhao2007d}
J. Zhao, Analytic continuation of multiple polylogarithms, \emph{Anal.\ Math.} \textbf{33}(2007), pp.\ 301--323.

\bibitem{Zhao2010}
J. Zhao, Linear relations of special values of multiple polylogarithms at roots of unity, \emph{Doc. Math.} \textbf{15}(2010), pp.\ 1--34.

\bibitem{Z2016}
J. Zhao, \emph{Multiple Zeta Functions, Multiple Polylogarithms and Their Special Values}, Series on Number
Theory and its Applications, Vol.~12, World Scientific Publishing Co. Pte. Ltd., Hackensack, NJ, 2016.

\bibitem{Z2024}
J. Zhao, Weighted and restricted sum formulas of Euler sums, \emph{Res. Number Theory} \textbf{10}(4)(2024), Paper No. 87, 40 pp.

\end{thebibliography}
\end{document}